\documentclass[12pt, twoside, leqno]{article}

\usepackage[T1]{fontenc}
\usepackage[utf8]{inputenc}
\usepackage{amsfonts}			
\usepackage{textcomp} 			
\usepackage[english]{babel}
\usepackage[normalem]{ulem}
\usepackage{graphicx,amsmath,amsthm,amssymb,latexsym,xcolor,stmaryrd, mathrsfs,mathtools}
\usepackage{hyperref,fancyhdr,dsfont,fourier-orns,wasysym,verbatim}
\usepackage{enumerate}
\everymath{\displaystyle}

\newtheorem{thm}{Theorem}[section]
\newtheorem{lema}[thm]{Lemma}
\newtheorem{cor}[thm]{Corollary}
\newtheorem{prop}[thm]{Proposition} 

\theoremstyle{definition}
\newtheorem{deff}[thm]{Definition}
\newtheorem{notation}[thm]{Notation}
\newtheorem{rmk}[thm]{Remark}

\newtheorem*{thmnonnum}{Theorem}
\newtheorem*{rmknonnum}{Remark}

\newtheorem*{defnonnum}{Definition}

\numberwithin{equation}{section}

\newcommand{\abs}[1]{\lvert #1\rvert}

\newcommand{\N}{\mathbb{N}}

\newcommand{\Z}{\mathbb{Z}}

\newcommand{\fie}{\varphi}
\newcommand{\ra}{\rightarrow}

\newcommand{\bigslant}[2]{{\raisebox{.2em}{$#1$}\left/\raisebox{-.2em}{$#2$}\right.}}
\newcommand*\quot[2]{{^{\textstyle #1}\big/_{\textstyle #2}}} 
\newcommand{\restr}[2]{#1_{\upharpoonright #2}}
\newcommand{\inv}[1]{#1\raisebox{1.15ex}{$\scriptscriptstyle-\!1$}}

\newcommand{\fct}[3]{\begin{array}[t]{lrcl}
#1\colon #2 \rightarrow #3\end{array}}
\newcommand{\accfonction}[5]{#1\colon \begin{cases}
&#2 \longrightarrow #3\\
&#4 \longmapsto #5
\end{cases}}

\DeclareMathOperator{\homeo}{Homeo}
\DeclareMathOperator{\homeomin}{Homeomin}
\DeclareMathOperator{\aut}{Aut}

\DeclareMathOperator{\supp}{supp}

\begin{document}


\baselineskip=17pt


\title{Strong orbit equivalence in Cantor dynamics and simple locally finite groups}

\author{Simon ROBERT\\
Institut Camile Jordan\\
Université Claude Bernard - Lyon1\\
43 Boulevard du 11 novembre 1928, 69622 Villeurbanne Cedex, France\\
E-mail : srobert@math.univ-lyon1.fr}

\date{}

\maketitle


\renewcommand{\thefootnote}{}

\footnote{2020 \emph{Mathematics Subject Classification}: Primary 37B02; Secondary 03E15}

\footnote{\emph{Key words and phrases}: Borel-reducibility, Cantor Dynamics, Minimal homeomorphism, Strong Orbit Equivalence, Kakutani-Rocklin partitions, Topological full group, isomorphism relation on the space of countable, locally finite simple groups.}

\renewcommand{\thefootnote}{\arabic{footnote}}
\setcounter{footnote}{0}


\begin{abstract}
We study certain countable locally finite groups attached to minimal homeomorphisms, and prove that the isomorphism relation on simple, countable, locally finite groups is a universal relation arising from a Borel $S_\infty$-action. This work also provides a dynamical approach to a result of Giordano, Putnam and Skau characterizing strong orbit equivalence.
\end{abstract}

\section{Introduction}

Algebraic tools (notably, dimension groups) have long been known as a useful and effective way to study minimal homeomophisms of the Cantor space; 
sometimes this leads to relatively complex, high-powered proofs of theorems that one would like to understand dynamically, so as
to gain a different perspective or extend them to other contexts. Such is one of the purposes of this article.
Using a different perspective can lead to new results: even though the core of this article is about topological dynamics, our main result is about Borel reducibility theory. We only briefly explain the general context here and refer the interested reader to {\cite{BecKec}}, section 3, {\cite{Gao}}, section 5 and references therein for more advanced results.

\begin{defnonnum}
Let $E,F$ be equivalence relations on standard Borel spaces. We say that  \emph{$E$ is Borel reducible to $F$}, and write $E\leq F$, if there exists a Borel map $f$ such that 
\[\forall x,x' \ xEx'\Longleftrightarrow f(x)Ff(x').\]
We call such a map $f$ a \emph{Borel reduction from $E$ to $F$}. If $E\leq F$ and $F\leq E$, we say that \emph{$E$ and $F$ are Borel bireducible}.
\end{defnonnum}

The idea behind this definition is that solving the classification problem associated to $E$ (i.e deciding when two elements are in the same $E$-class) is then simpler than solving the one associated to $F$, hence the relation $F$ can be considered "more complex" than $E$.
However, without a requirement on $f$, this notion would only detect the number of equivalence classes, and such a map could be very chaotic, this is why one wants to ensure that the correspondence is somehow computable, and in this paper, the notion of computable retained is Borel.

An important source of equivalence relations is given by actions of Polish groups. The following theorem shows that given a Polish group, there always exists an equivalence relation arising from it that is as complex as possible :

\begin{thmnonnum}[{\cite{BecKec}, Corollary 3.5.2}]
Let $G$ be a Polish group. There exists an equivalence relation $E_G$ arising from a Borel $G$-action on a standard Borel space such that any other such relation Borel reduces to it.
\end{thmnonnum}
We will call such a relation \emph{$G$-universal}. We will be interested in equivalence relations arising from actions of $S_\infty$, the permutation group of a countably infinite set, which is a well known Polish group (see {\cite{Gao}, section 2.4}), and we will denote $E_\infty$ a (unique up to Borel bireducibility) universal relation arising from an action of it. Our main theorem is the following:
\begin{thmnonnum}
The relation of isomorphism of countable, locally finite, simple groups is a universal relation arising from a Borel action of $S_\infty$ (i.e is Borel bireducible to $E_\infty$).
\end{thmnonnum}
A proof of this result is given in section \ref{section: Borel complexity}.  We use a result of {\cite{Mel20}}, namely that strong orbit equivalence (a notion from topological dynamics that we will define straight after this) is Borel bireducible to $E_\infty$, and we find a Borel way to associate to minimal homeomorphisms some locally finite simple groups that are isomorphic exactly when the homeomorphisms are strong orbit equivalent.

In order to describe more precisely the groups at stake, let us move on to some notions of dynamics. 

Any action of a group $G$ on a set $X$ induces an equivalence relation, whose equivalence classes are the $G$-orbits.
Forgetting the action to focus only on this equivalence relation, two actions can be very different
dynamically speaking, even coming from different groups, yet generate the same equivalence relation up to a bijection. 
This leads to the notion of orbit equivalence:
\begin{defnonnum}
Let $G,G'$ be groups and $X,X'$ be sets. Two actions $\alpha\colon G\curvearrowright X$ and $\beta\colon G'\curvearrowright \nolinebreak X'$ are \emph{orbit equivalent} if there exists a bijection \fct{h}{X}{X'} (called \emph{an orbit equivalence between $\alpha$ and $\beta$}) that realizes a bijective correspondence between $\alpha$-orbits and $\beta$-orbits, i.e $$\forall x\in X \ h(Orb_\alpha(x))=Orb_\beta(h(x))$$
\end{defnonnum}
\noindent If the sets $X$ and $X'$ are endowed with a structure, it is natural to require $h$ to preserve this structure, i.e. to be an 
isomorphism from $X$ to $X'$, and this will be a part of our conventions below. It is then tempting to try and classify actions up to orbit equivalence. In a measure theoretical context, this is now a well-understood problem; 
A combination of works by Dye ({\cite{Dye}}) and later Ornstein--Weiss ({\cite{OW}}) gives the following famous result: there is only one probability measure preserving, ergodic action of an amenable group on a standard probability space up to orbit equivalence.

However, in a topological context, when $X$ is a Cantor space, the situation is more complicated. In {\cite{GPS1}} and {\cite{GPS2}}, Giordano, Putnam and Skau managed, 
using $C^*$-algebra techniques, to obtain 
results about $\Z$-actions. Before quoting them, let us recall some basic objects from topological dynamics:
\begin{defnonnum}
A homeomorphism $\fie$ on the Cantor space $X$ is said to be \textit{minimal} if every $\fie$-orbit is dense in $X$: $\forall x\in X \ \overline{Orb_\fie(x)}=X$.
\end{defnonnum}
\begin{defnonnum}
The \textit{full group of $\fie$}, denoted by $[\fie]$, consists of all homeomorphisms that preserve $\fie$-orbits:
\[[\fie]=\left\{f\in\homeo(X) \colon \forall x\in X \  \exists n_f(x)\in\Z \text{ such that } f(x)=\fie^{n_f(x)}(x)\right\} \]
For $f\in [\fie]$, the application $\accfonction{n_f}{X}{\Z}{x}{n_f(x)}$ is called \textit{a cocycle of $f$}.
\end{defnonnum}
\begin{rmknonnum}
Since $\fie$ is aperiodic (because it is minimal), cocycles are uniquely defined.
\end{rmknonnum}
\begin{defnonnum}
The \textit{topological full group of $\fie$}, denoted by $\llbracket\fie\rrbracket$, is the subgroup of elements in $[\fie]$ which have a continuous cocycle.
\end{defnonnum}
Using this vocabulary, Giordano, Putnam and Skau's Theorem about orbit equivalence is the following:
 
\begin{thmnonnum}[{\cite{GPS2}, Corollary 4.6}]
Let ($\fie_1,X_1)$ and $(\fie_2,X_2)$ be minimal Cantor systems. Then the following are equivalent:
\begin{enumerate}[(i)]
    \item The two systems ($\fie_1,X_1)$ and $(\fie_2,X_2)$ are orbit equivalent
    \item The full groups $[\fie_1]$ and $[\fie_2]$ are isomorphic as abstract groups
    \item \label{GPS2-iii} There exists a homeomorphism $\fct{g}{X_1}{X_2}$ that pushes forward $\fie_1$-invariant measures onto $\fie_2$-invariant measures, 
    i.e $g_*M(\fie_1)=M(\fie_2)$, where 
    $$M(\fie_i)=\{ \mu \text{ probability measure on } X_i \colon {\fie_i}_*\mu=\mu \}$$
\end{enumerate}   
\end{thmnonnum}
In particular, point (\ref{GPS2-iii}) implies that there exist continuum many pairwise non orbit equivalent Cantor minimal systems, 
see for example {\cite{Dow}}.

Giordano--Putnam--Skau also considered a notion which differs slightly from orbit equivalence and that we are concerned with in this article, which they call \emph{strong orbit equivalence}. 
To explain it, let us introduce some terminology: an orbit equivalence $h$ between two minimal systems $(\fie,X)$ and $(\psi,X)$ gives rise
to two applications $n$ and $m$ from $X$ to $\Z$ such that
$$\forall x \in X \ h(\fie(x))=\psi^{n(x)}(h(x)) \text{ and conversely }\inv{h}(\psi(x))=\fie^{m(x)}(\inv{h}(x)).$$
These two applications are called \emph{cocycles associated to $h$}. The orbit equivalence $h$ is called a \emph{strong orbit equivalence} 
if both cocycles $n$ and $m$ have at most one point of discontinuity. This notion is, as the next theorem highlights, closely related to 
the group $\Gamma^\fie_{x}$ (associated to a minimal system $(\fie,X)$ and a point $x\in X$) consisting of all elements $\gamma$ of 
$\llbracket\fie\rrbracket$ such that $\gamma(Orb^+_\fie(x))=Orb^+_\fie(x)$.
\begin{thmnonnum}[{\cite{GPS2}, Corollary 4.11}]
Two Cantor minimal systems $(\fie,X)$ and $(\psi,X)$ are strong orbit equivalent if and only if $\Gamma^\fie_{x}$ and $\Gamma^\psi_{y}$ are isomorphic as abstract
groups for all $x,y\in X$.
\end{thmnonnum}
An elementary proof of this result, which justifies our original interest for the groups $\Gamma_x^\fie$, is given in section \ref{section: GPS proof}. By \emph{elementary}, 
we mean that the proof is only based on manipulations on closed-open sets.
The objects we mainly use are sequences of Kakutani-Rokhlin partitions, the point of which is to describe in finite time the image of (almost) any closed-open set by a minimal homeomorphism (details in section \ref{subsection: KR partitions}). Our general strategy is based on a theorem of Krieger ({\cite{Kr}, Theorem 3.5}). Recall that
two subgroups $\Gamma$ and $\Lambda$ of $ \homeo(X)$ are \emph{spatially isomorphic} if there exists a homeomorphism $g$ of $X$ such that 
$$\Gamma=g\Lambda\inv{g}.$$ Krieger's theorem establishes that two countable, locally finite groups $\Gamma$ and $\Lambda$ satisfying some extra properties, called \emph{ample groups},
are spatially isomorphic if and only if there exists a homeomorphism $g$ such that 
$$\accfonction{\overline{g}}{\quot{CO(X)}{\Gamma}}{\quot{CO(X)}{\Lambda}}{Orb_\Gamma(A)}{Orb_\Lambda(g(A))}$$
is a bijection. 

To set the stage for the application of Krieger's theorem, we first recall the construction of the groups $\Gamma_x^\fie$
out of a sequence of Kakutani-Rokhlin partitions, which shows that they are locally 
finite (and actually even ample). Afterwards, an important step is to show that the relation induced by the $\Gamma_x^\fie$ on closed-open sets does not depend
on the point $x$, i.e using Krieger's vocabulary that they have the same \emph{dimension range}.
We also notice that this approach recovers the result that $\overline{\Gamma^\fie_x}=\overline{\llbracket\fie\rrbracket}$ (see {\cite{IM}, Theorem 5.6}, or {\cite{GM}})

Finally, the groups $\Gamma_x^\fie$ are not necessarily simple, but in section \ref{section: Borel complexity}, adapting arguments from {\cite{BM}} on the one hand, and {\cite{Med}}
on the other hand, we show respectively that the groups $D(\Gamma_x^\fie)$ are always simple, and that any isomorphism between two such groups is spatial, which makes them locally finite, simple groups attesting to strong orbit equivalence of homeomorphisms they are attached to, as desired.

\section{Preliminaries}\label{section: Preliminaries}
For the remainder of this article, $X$ denotes the Cantor space; $\fie$ denotes a minimal homeomorphism of $X$ and $x_0$ a point of $X$; we call a 
closed-open set a clopen set, and $CO(X)$ stands for the Boolean algebra of clopen subsets of $X$. Moreover, we use the notation $\llbracket i;j\rrbracket$ to denote the interval of integers $\{i,i+1\ldots,j\}$, which is by convention empty if $i>j$. Finally we draw the reader's attention to the following fact: for us, $\N$ is the set
of non-negative integers, and $Orb_\fie^+(x_0)=\left\{\fie^n(x_0)\right\}_{n\in\N}$.

Most of this section consists of well-known facts and objects, and the reader knowing what a Kakutani-Rokhlin partition is can skim through it until Theorem \ref{Krthm} and discussion below.
\medbreak

The following property is very useful to build and understand elements of a topological full group in practice:
\begin{prop}
Let $f$ be a homeomorphism of $X$. Then $f\in\llbracket\fie\rrbracket$ if and only if there exist $A_1,\ldots ,A_n\in CO(X)$
and $k_1,\ldots,k_n\in\Z$ such that $X=\bigsqcup_{i=1}^n A_i$ and $\restr{f}{A_i}=\restr{\fie^{k_i}}{A_i}$.
\end{prop}

The following theorem is well-known (and easy to check in this situation):
\begin{thm}\label{stone}(Stone's representation theorem for Boolean algebras)

For every automorphism $\fct{\alpha}{CO(X)}{CO(X)}$, there exists a (unique) homeomorphism $g$ on $X$ that extends 
$\alpha$ (i.e $g(C)=\alpha(C)$ for every $C\in CO(X)$)
\end{thm}
\begin{rmk}
$\aut(CO(X))$ is a closed subset of $CO(X)^{CO(X)}$ equipped with the product topology (of the discrete topology on $CO(X)$). Thus this theorem implies that $\homeo(X)$ inherits a Polish group structure from that topology; a sub-basis is given by the sets $$[A\ra B]=\{f\in \homeo(X)\colon f(A)=B\}$$ where $A$ and $B$ are clopen sets.
Actually, this is the only Polish topology on $\homeo(X)$ (see {\cite{Ros-Sol}}).
We always consider $\homeo(X)$ as endowed with this topology, and think of it in this way, even though one can use a complete metric on $X$ and think of $\homeo(X)$ as being endowed with the topology of uniform convergence.
\end{rmk}

\subsection{Kakutani-Rokhlin partitions}\label{subsection: KR partitions}

Kakutani-Rokhlin partitions (for which we will use the term "K-R partitions" for brevity) are an essential tool for an elementary study of the dynamics of minimal homeomorphisms.

The idea is the following: given a nonempty clopen set $A$, we look at the first return map
$$\accfonction{\fie_A}{A}{A}{x}{\fie^{\tau_A(x)}(x)} \ , \text{ where}$$ 
$$\accfonction{\tau_A}{A}{\N}{x}{\min\{k>0\colon \fie^k(x)\in A\}}$$
is well-defined because a homeomorphism is minimal if and only if every forward orbit is dense in $X$. We can also consider $\fie_A$ as a homeomorphism of $X$ by requiring $\restr{{\fie_A}}{A^c}=\restr{id}{A^c}$. Since $A$
is compact and $\tau_A$ is continuous, $\tau_A(A)$ is finite, and $A$ admits a finite clopen partition consisting of points that return to $A$ in a certain
number of steps. 
\begin{figure}[h]
    \centering
    \includegraphics[scale=0.4]{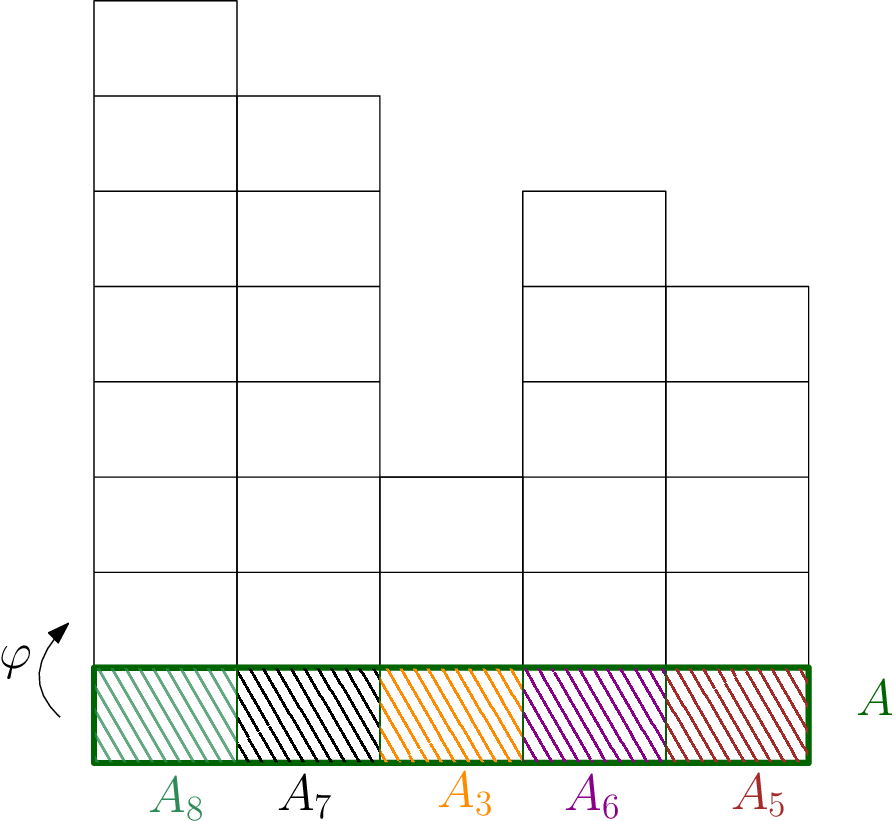}
    \caption{K-R partition built from $A$}
    \label{KR-partition}
\end{figure}
More formally, $A=\bigsqcup_{k\in\N}A_k$, where $A_k := \inv{\tau_A}({k})$ is clopen, and all but a finite number of them are empty. 
We then obtain a partition of $X$ that can be represented as on figure \ref{KR-partition}.

More generally we define K-R partitions as follows:

\begin{deff}
A \textit{K-R partition associated to $\fie$} is a clopen partition 
$$\Xi=(D_{i,j})_{i\in\llbracket1,N\rrbracket, 
j\in\llbracket0,H_i-1\rrbracket}$$ 
of $X$ such that $\fie(D_{i,j-1})=D_{i,j}$ for all $i\in\llbracket1,N\rrbracket$ and all $j\in\llbracket1,H_i-1\rrbracket$.

We call, for $i\in\llbracket1,N\rrbracket$, \textit{tower $i$ of $\Xi$} the set $T_i=\bigsqcup_{j\in\llbracket0,H_i-1\rrbracket}D_{i,j}$,
and \textit{height} of this tower the number $H_i$. We also talk about the \textit{$k$-th floor} of $\Xi$ to designate
$\bigsqcup_{i\in\llbracket1,N\rrbracket}\fie^k(D_{i,0})$. The 0-th floor is called \textit{the base} of $\Xi$, and denoted by $B(\Xi)$,
and the $-1$-th one is called \textit{the top} of the partition.

In the case where the partition carries an index ($\Xi_n$ instead of $\Xi$), which will be the case very soon, we talk about the tower $T^n_i$, the atom $D^n_{i,j}$, etc.
\end{deff}

\begin{rmk}
Note that $\bigsqcup_{i\in\llbracket1,N\rrbracket}D_{i,H_i-1}$ is mapped by $\fie$ onto $B(\Xi)$, and thus is the top of the partition. 
However, $D_{i,H_i-1}$ can be mapped anywhere in $B(\Xi)$.
\end{rmk}

\begin{notation}
We denote by $\langle \Xi \rangle$ the Boolean algebra generated by the atoms \newline $(D_{i,j})_{i\in\llbracket1,N\rrbracket, j\in\llbracket0,H_i-1\rrbracket}$ of $\Xi$.
\end{notation}

The idea behind K-R partitions is that they represent how $\fie$ acts on the atoms of $\Xi$ which are not contained in the top of $\Xi$.
We want to get better and better approximations of $\fie$ in terms of how it acts on clopen sets. In the  same spirit as Theorem \ref{stone}, one can check 
that knowing how $\fie$ acts on clopen sets that does not contain a 
particular point determines $\fie$. In our case, we want to construct a sequence of K-R partitions in which every clopen set appears,
and such that the intersection of their bases (and hence of their tops) is a single point.
To do that, we have to be able to refine a K-R partition into another one, without losing information we have already obtained.
This is the purpose of the following proposition:

\begin{prop}\label{prop: raffinement d'une KR-part contenant un clopen fixé}
Let $A\in CO(X)$, and $\Xi$ be a K-R partition. Then there exists a K-R partition $\Xi'$ finer than $\Xi$ such that $A\in  \langle\Xi' \rangle$.
\end{prop}

\begin{figure}[h]
    \centering
    \includegraphics[scale=0.4]{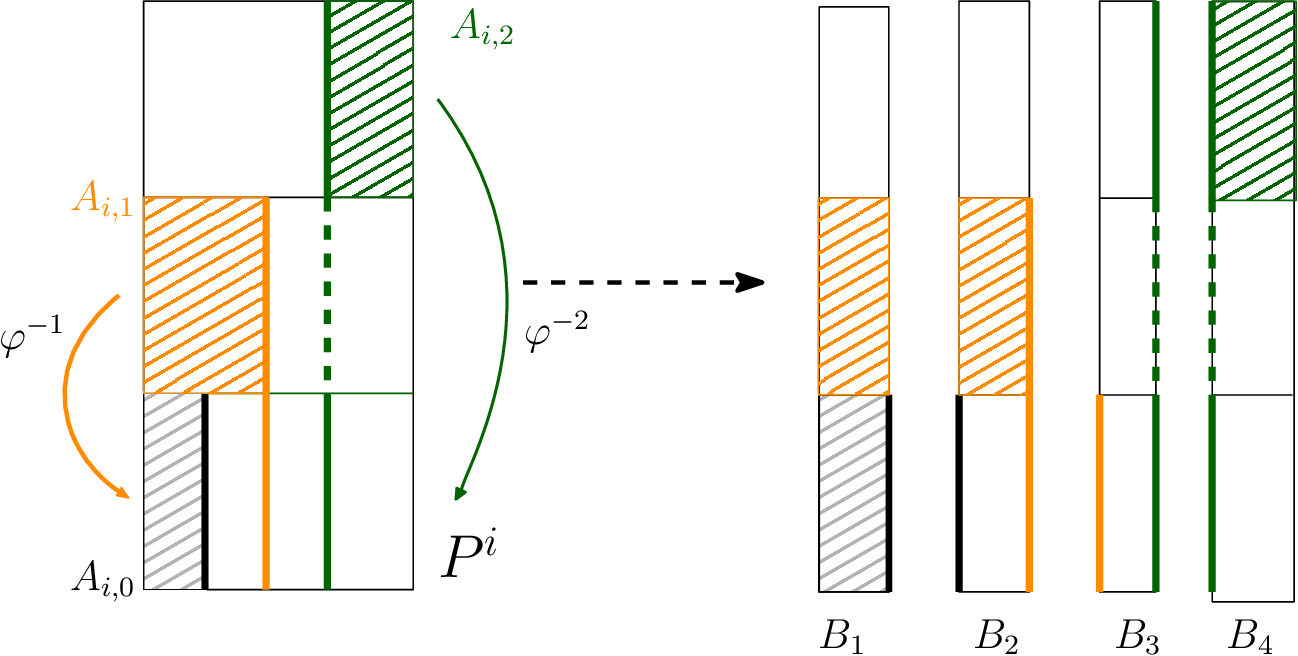}
    \caption{Cutting process in tower $i$ of $\Xi$, $A_{i,j}=A\cap D_{i,j}$}
    \label{raffinementclopen}
\end{figure}
\begin{proof}
For each tower $T_i$ of $\Xi$, we define an equivalence relation $\mathscr{R}_i$ on $D_{i,0}$ by
\[x\mathscr{R}_i y \Leftrightarrow  \forall j<H_i \ \left( (\fie^j(x)\in A \wedge \fie^j(y)\in A) \vee (\fie^j(x)\in A^c \wedge\fie^j(y)\in A^c) \right)\] 

Denote by $P^i$ the partition of $D_{i,0}$ associated to $\mathscr{R}_i$. (cf figure \ref{raffinementclopen})

Denoting by $(B_l)_{l\in\llbracket0,N_i\rrbracket}$ the atoms of $P^i$, we just "cut" the $i-th$ tower into $N_i$ towers whose bases are the $B_l$'s.
Following this method for each tower, we obtain a K-R partition $\Xi'$ finer than $\Xi$ such that every $A\cap D_{i,j}$ is in $ \langle\Xi' \rangle$, 
and so $A$ is also in $ \langle\Xi' \rangle$.
\end{proof}

Let $(U_n)_{n\in\N}$ be a clopen basis  of the topology. Applying Proposition \ref{prop: raffinement d'une KR-part contenant un clopen fixé}
inductively, we obtain the following:
\begin{cor}\label{KR-properties}
 There exists a sequence of K-R partitions $(\Xi_n)_{n\in\N}$ such that the following conditions hold for all $n\in\N$
 :
 \begin{enumerate}[(i)]
    \item\label{KR-properties1} $\Xi_{n+1}$ is finer than $\Xi_n$
    \item\label{KR-properties2} $B(\Xi_{n+1})\subset B(\Xi_n)$ and $\bigcap_i B(\Xi_i)=\{x_0\}$ 
    \item\label{KR-properties3} $U_n\in  \langle\Xi_n \rangle$
    \item\label{KR-properties4} The minimal height of $\Xi_n$ is greater than $n$.
\end{enumerate}
\end{cor}
When we consider a sequence of K-R partitions, we assume from now on that it fulfills properties \eqref{KR-properties1} to \eqref{KR-properties4}.

\begin{rmk}\label{rmk: finite images of a small neighbourhood are disjoint}
Property \eqref{KR-properties4} is obtained thanks to aperiodicity of $\fie$: \newline 
the points $x_0, \fie(x_0),\ldots,\fie^n(x_0)$ are all distinct, thus for a sufficiently small neighbourhood $U$ of $x_0$, we have that
$U,\fie(U),\ldots,\fie^n(U)$ are pairwise disjoint.
\end{rmk}
 
\begin{deff}
We call \textit{base point} of a sequence of K-R partitions $(\Xi_n)$ the point $x_0$ that appears in property \eqref{KR-properties2}, and \textit{top point} of the 
sequence the point $\inv\fie(x_0)$.
\end{deff}

\begin{rmk}
An important thing to understand about this construction is that a tower of $\Xi_{n+1}$ is obtained by cutting (vertically) the towers of $\Xi_n$
and then stacking some of them on top of each other. This is what is called "cutting and stacking" (see figure \ref{gamma_2_partitions}; arrows tell us
where a clopen set at the top of the partition is sent by $\fie$. If there is none, it means that it is sent into the base of $\Xi_{n+1}$).
Indeed, for every $i$, $D_{i,0}^{n+1}\subset D_{i_0,0}^n\subset B(\Xi_n)$ for some $i_0$, so $\bigsqcup_{k=0}^{H_{i_0}^n-1} D^{n+1}_{i,k}$ is exactly obtained by cutting
the $i_0$-th tower of $\Xi_n$. If $D_{i,H_{i_0}^n-1}^{n+1}$ is not on the top of $\Xi_{n+1}$, it suffices to look at the tower of $\Xi_n$ in which it 
is mapped by $\fie$ to know which "tower" to stack over this one: if $\fie(D^{n+1}_{i,H_{i_0}^n-1})\subset D_{i_1,0}^n$, then 
$\bigsqcup_{k=H_{i_0}^n}^{H_{i_0}^n+H_{i_1}^n-1} D^{n+1}_{i,k}$ is obtained by cutting the $i_1$-th tower of $\Xi_n$, and so on.
\end{rmk}

\subsection{A locally finite group associated to a minimal homeomorphism}

In this section we study the group of homeomorphisms in $\llbracket\fie\rrbracket$ that preserve the non-negative semi-orbit 
of $x_0$. We denote this group 
$$\Gamma_{x_0}^\fie:=\{f\in\llbracket\fie\rrbracket \colon f(Orb^+(x_0))=Orb^+(x_0)\} \ ,$$ 
or just $\Gamma_{x_0}$ if the homeomorphism involved is clear from the context.
As recalled in the introduction, these groups are used in {\cite{GPS2}} and are one of the main objects under consideration in our article. 
For the moment we focus on explaining how they are related to K-R partitions.

\begin{figure}[ht]
    \centering
    \includegraphics[width=\linewidth]{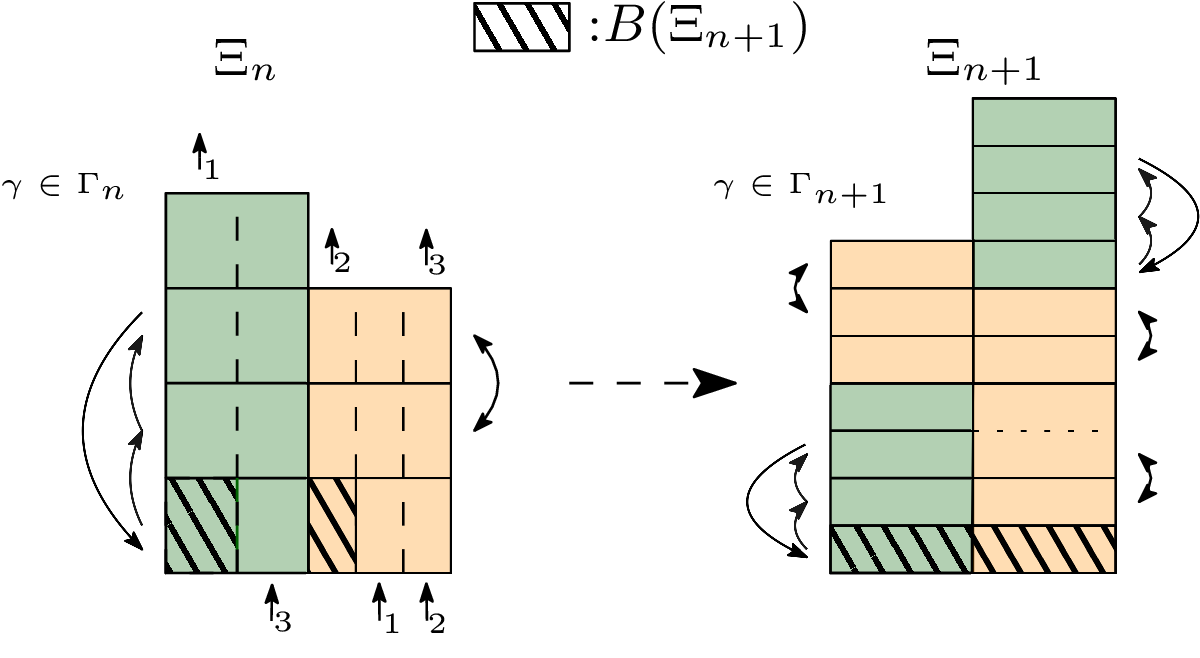}
    \caption{an element $\gamma$ in $\Gamma_n$ and in $\Gamma_{n+1}$}
    \label{gamma_2_partitions}
\end{figure}
Let $(\Xi_n)_{n\in\N}$ be a sequence of K-R partitions (satisfying properties \eqref{KR-properties1} to \eqref{KR-properties4}). For each $n\in\N$, we consider the group $\Gamma_n$ 
of elements of $\llbracket\fie\rrbracket$ that have a constant cocycle on atoms of $\Xi_n$ and that "stay" in each tower of the partition, meaning 
that an atom can not go through either the top of its tower nor its base under the action of a $\gamma$ in $\Gamma_n$.

More formally, if $(D_{i,j})_{i,j}$ are the atoms 
of $\Xi_n$ and $H_i$ is the height of the $i$-th tower, then
\[ \Gamma_n=\{f\in\llbracket\fie\rrbracket \colon \forall i,j \ \restr{f}{D_{i,j}}=\restr{\fie^k}{D_{i,j}},~ k\in\llbracket-j,H_i-j-1\rrbracket\}\]

An element of this group may be thought of as permuting atoms in each tower, but it is important to remember that it acts on points,
in particular that helps to see the inclusion $\Gamma_n\subset\Gamma_{n+1}$. 

Indeed, as $\Xi_{n+1}$ is constructed from $\Xi_n$ by cutting and stacking, each $\gamma\in\Gamma_n$ also belongs to $\Gamma_{n+1}$ 
(see figure \ref{gamma_2_partitions})
\begin{deff}
We have obtained from a sequence of K-R partitions $(\Xi_n)$ a locally finite group $\Gamma_\Xi=\bigcup_{n\in\N}\Gamma_n$. 
If there is risk of confusion, we write $\Gamma^\fie_\Xi$ to mention the homeomorphism involved.
\end{deff} 

For the moment, it seems that $\Gamma_\Xi$ depends on the sequence of K-R partitions we have chosen. We 
note that it only depends on the choice of the base point. Recall that $$\Gamma_{x_0}^\fie=\left\{\gamma\in \llbracket\fie\rrbracket\ |\ \gamma(Orb_\fie^+(x_0))=Orb_\fie^+(x_0)\right\}$$.

\begin{prop}\label{gamma_egal_orbit_pos}
The groups $\Gamma_\Xi$ and $\Gamma_{x_0}$ are equal. 
\end{prop}

\begin{proof}

Let $\gamma\in\Gamma_\Xi$, and $y\in Orb^+(x_0)$. Say $y=\fie^k(x_0)$ for a $k\in\N$. By property \eqref{KR-properties4} of Corollary \ref{KR-properties}, there exists $N\in\N$ big
enough that $\gamma\in\Gamma_N$ and each tower of $\Gamma_N$ has an height greater or equal to $k+1$.
Then $x_0$ and $y$ belong to the same tower of $\Xi_N$, and $x_0\in B(\Xi_N)$. Since $\gamma\in\Gamma_N$, $\gamma(y)=\fie^j(y)$ with $j\geq -k$, 
we have $\gamma(y)\in Orb^+(x_0)$. Since the same property holds for $\inv{\gamma}$, we obtain 
$$\gamma(Orb^+(x_0))=Orb^+(x_0).$$

Conversely, let $h\in\Gamma_{x_0}$. Let $X=\bigsqcup_{i=1}^n A_i$ be a clopen partition such that $\restr{h}{{A_i}}=\restr{\fie^{k_i}}{{A_i}}$.
Define $m_i=\min\{k\geq 0 \colon \fie^{k}(x_0)\in A_i\}.$ Necessarily, $k_i\geq -m_i$. On the other hand, since $\fie^k(x_0)\notin A_i$ for all 
$k\in\llbracket0,m_i-1\rrbracket$, there exists a clopen set $U_k^i\ni\fie^k(x_0)$ disjoint from $A_i$. Thus 
$\bigcap_i\bigcap_{0\leq k<m_i}\fie^{-k}(U_k^i)$ is a neighbourhood of $x_0$, hence $B(\Xi_N)$ is contained in it for $N$ big enough.
Moreover, we can take $N$ big enough that every $A_i\in \langle\Xi_N \rangle$, and every tower in $\Xi_N$ has an height greater than $\max_i(m_i)$. 
Then every atom of $\Xi_N$ is included into one of the $A_i$'s, and an atom contained in $A_i$ cannot appear before the $m_i$-th floor 
(because the $k$-th floor is a subset of $U_k^i$ which is disjoint from $A_i$).

Since $h\in\llbracket\fie\rrbracket$, we also have 
$$h(Orb_\fie^{<0}(x_0))=Orb_\fie^{<0}(x_0)$$
(where $Orb_{\fie}^{<0}(x_0)$ stands for $\{\fie^k(x_0), k<0\}$). Thus, 
defining 
$$m'_i=\max\{k<0 \colon \fie^{k}(x_0)\in A_i\}$$
and using the same argument, we get that $k_i<-m'_i$ and an atom contained in $A_i$ cannot appear higher than the 
$m'_i$-th floor (i.e there is at least $\abs{m'_i}-1$ atoms above it before reaching the top of the tower).

This shows that $h$ cannot send an atom of $\Xi_N$ through the top or the base of the partition, and so $h\in\Gamma_\Xi$.
\end{proof}

In the next section, we use as a key ingredient the following theorem which is a combination of two results: the equivalence between Conditions \ref{thm Krieger, cond 2} and \ref{thm Krieger, cond 3} is due to Krieger ({\cite{Kr}, Theorem 3.5}), and its proof is elementary (actually, it is based on a back-and-forth argument that is easy 
to set up directly in the case of the groups $\Gamma_{x}^\fie$; we chose not to include this proof in an attempt not to overextend our claims to the reader's attention).
The equivalence between Conditions \ref{thm Krieger, cond 1} and \ref{thm Krieger, cond 2} is known since Giordano, Putnam and Skau (see {\cite{GPS2}, Theorem 4.2}) in the case $(\Gamma,\Lambda)=(\Gamma^\fie_{x_1},\Gamma^\psi_{x_2})$. New proofs based on reconstruction theorems have been developed later, and we give a proof in the case $(\Gamma,\Lambda)=(D(\Gamma^\fie_{x_1}),D(\Gamma^\psi_{x_2}))$ using this approach in Proposition \ref{prop: every isom is spatial}.

\begin{thm}\label{Krthm}
Let $\fie$,$\psi$ be minimal homeomorphisms, $x_1,x_2$ be two points in $X$, and $(\Gamma,\Lambda)$ denote either $(\Gamma^\fie_{x_1},\Gamma^\psi_{x_2})$ or their commutator subgroups $(D(\Gamma^\fie_{x_1}),D(\Gamma^\psi_{x_2}))$. Then the following are equivalent: 
\begin{enumerate}
    \item \label{thm Krieger, cond 1} The abstract groups $\Gamma$ and $\Lambda$ are isomorphic
    \item \label{thm Krieger, cond 2} There exists a homeomorphism $g$ such that $\Gamma=g\Lambda\inv{g}$
    \item \label{thm Krieger, cond 3} There exists a homeomorphism $g$
such that $$\accfonction{\overline{g}}{\bigslant{CO(X)}{\Gamma}}{\bigslant{CO(X)}{\Lambda}}{Orb_{\Gamma}(A)}{Orb_{\Lambda}(gA)} \text{  is a bijection.}$$ 

\end{enumerate}
\end{thm}

\begin{rmk}\label{rmk : more than original Krieger is true}
Moreover, a bit more is true : if Condition \ref{thm Krieger, cond 3} is satisfied, and $(x,y),(x',y')$ are two pairs of point belonging respectively to different $\Lambda$ and $\Gamma$ orbits, then the homeomorphism $g$ of Condition \ref{thm Krieger, cond 2} can be chosen such that $g(x)=x'$ and $g(y)=y'$. This is not part of the original theorem in {\cite{Kr}} but can easily be seen following the proof. A proof of a much stronger generalization is given in {\cite{MR} (Theorem 3.11)}.
\end{rmk}

\begin{deff}[Krieger's vocabulary]\label{def: dimension range}
The relation induced by a group $\Gamma^\fie_x$ on $CO(X)$ is called \emph{dimension range of $\Gamma^\fie_x$}. A homeomorphism $g$ such that $\overline{g}$
(defined as in \ref{thm Krieger, cond 3}) is a bijection is \emph{an isomorphism between the dimension ranges of $\Gamma^\fie_x$ and $\Gamma^\psi_y$}, 
and those two groups are said to have \emph{isomorphic dimension ranges}.
\end{deff}

\section{An elementary proof of Giordano, Putnam and Skau's characterization of strong orbit equivalence}\label{section: GPS proof}

We are now ready to prove Giordano, Putnam and Skau's characterization of strong orbit equivalence. Let us recall the definition of strong orbit equivalence \nolinebreak:
\begin{deff}
Two minimal homeomorphisms $\fie$ and $\psi$ are called \emph{strong orbit equivalent} if there exists a homeomorphism $g$ such that 
$$\forall x\in X \  g(Orb_\fie(x))=Orb_\psi(g(x))$$
and such that the two associated cocyles $n$ and $m$, defined by 
 $$\forall x\in X \  g(\fie(x))=\psi^{n(x)}(g(x)) \text{ and } \inv{g}(\psi(x))=\fie^{m(x)}(\inv{g}(x))$$
have at most one point of discontinuity each.
\end{deff}

First of all we need to  gain a better understanding of the relation induced by $\Gamma_{x}$ on clopen sets (the dimension range of $\Gamma_x$, 
cf definition \ref{def: dimension range}). We have seen that the group does not depend on the sequence of partitions out of which it is 
constructed. Now we go further, showing that the dimension range of $\Gamma_x$ does not depend on the base point $x$ either. 

\begin{lema}\label{lemma GPS SOE-changing base point does not change dimension range}
For every $x,x'\in X$, $\Gamma_{x}$ and $\Gamma_{x'}$ have the same dimension range.
\end{lema}
\begin{figure}[ht]
    \centering
    \includegraphics[width=\linewidth]{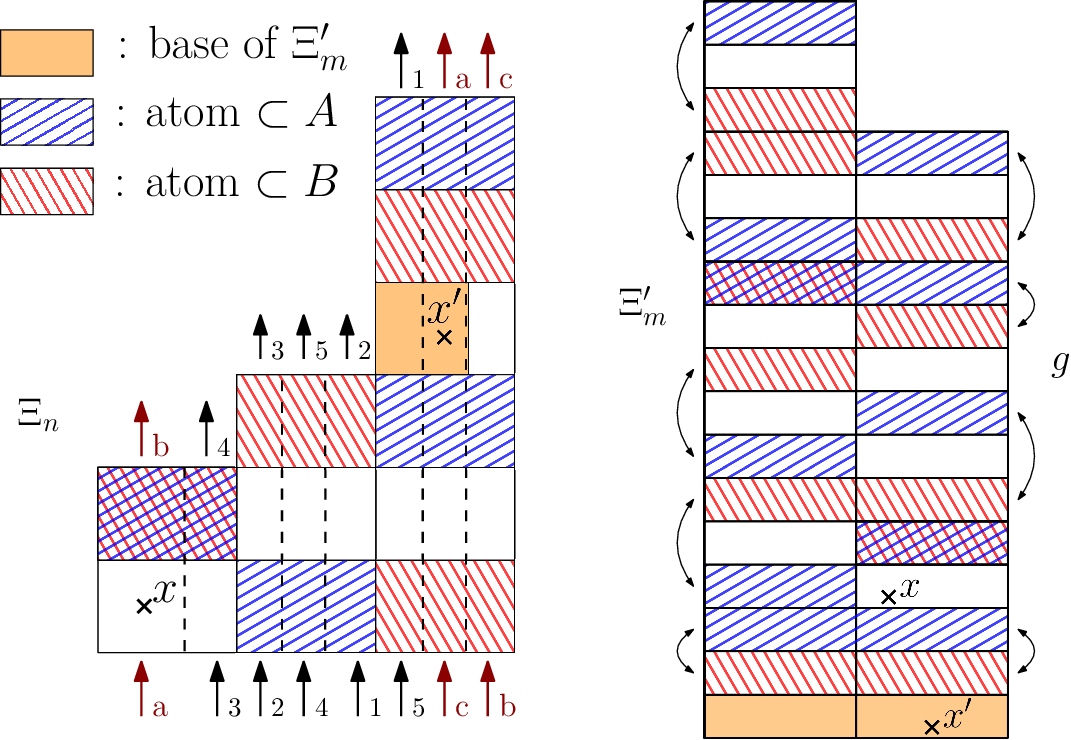}
    \caption{Key figure for understanding of Lemma \ref{lemma GPS SOE-changing base point does not change dimension range}}
    \label{chgtptbase_fig}
\end{figure}

\begin{proof}
Let $(\Xi_n)$ and $(\Xi_n')$ be two sequences of K-R partitions, with bases $(D^n)$ and $(F^n)$ respectively, such that $\bigcap D^n = \{x\}$ and 
$\bigcap F^n = \{x'\}$.
Let also $A$ and $B$ be two clopen sets such that $B\in Orb_{\Gamma_\Xi}(A)$. There exists $n\in\N$ and $\gamma\in\Gamma_{n,\Xi}$ such that
$\gamma(A)=B$, and $A,B$ are in $\langle\Xi_n\rangle$. Then in each tower of $\Xi_n$ there are as many atoms which are contained in $A$ 
as in $B$. Suppose $x'\in D^n_{i,j}$, and choose $m\in\N$ such that $\Xi_m'$ refines $\Xi_n$, and $F^m\subset D^n_{i,j}$.
Let $T'_k$ be a tower of $\Xi_m'$ with base $F^m_{k,0}$. 
Then 
$$H'^{m}_k=H^n_i-j+H^n_{i_1}+\ldots+H^n_{i_t}+j \text{ for some indices } i_1,\ldots,i_t$$ and $T'_k$ is composed of a stacking of 
$\bigcup_{p\geq j}D^n_{i,p}\cap T'_k$, $T^n_{i_1}\cap T'_k$,$\ldots$, $T^n_{i_t}\cap T'_k$, and eventually $\bigcup_{p<j}D^n_{i,p}\cap T'_k$
(see figure \ref{chgtptbase_fig}). So there are as many atoms of $T'_k$ that are included in $A$ as in $B$. 
We can then naturally define an involution $g\in\Gamma_{\Xi_m'}$ such that $g(A)=B$ (see figure \ref{chgtptbase_fig}); whence $B\in Orb_{\Gamma_{\Xi'}}(A)$, which concludes the proof.
\end{proof}

\begin{rmk}\label{rmk: lien entre mes lemmes et la densité de gamma dans le groupe plein topo}
This lemma could also be seen as an easy consequence of the fact that $\overline{\Gamma_\fie}=\overline{\llbracket\fie\rrbracket}$ 
(see \cite{GM} or \cite{IM}, Theorem 5.6). Conversely, Lemmas 
\ref{lemma GPS SOE-changing base point does not change dimension range} and \ref{lemma GPS SOE-be in Gamma orbit can be seen piecewise} give an elementary proof of that fact, using very similar arguments to those in the proof of Theorem \ref{gps1}.
\end{rmk}

\begin{thm}[{\cite{GPS2}, Corollary 4.11}]\label{gps1}
Let $\fie$ and $\psi$ be two minimal homeomorphisms. The following are equivalent:
\begin{enumerate}
 \item $\fie$ is strong orbit equivalent to $\psi$
 \item There exist $x,y \in X$ such that $\Gamma^\fie_{x}$ and $\Gamma^\psi_{y}$
are isomorphic as abstract groups \label{condition 2}
 \item For all $x,y \in X$, $\Gamma^\fie_{x}$ and $\Gamma^\psi_{y}$
are isomorphic as abstract groups. \label{condition 3}
\end{enumerate}
\end{thm}

\begin{rmk}\label{rmk: abstract isomorphism is equivalent to isomorphism between dimension ranges}
Thanks to Theorem \ref{Krthm}, for $x,y\in X$ the condition ``$\Gamma^\fie_{x}$ and $\Gamma^\psi_{y}$
are isomorphic as abstract groups'' is equivalent to ``$\Gamma^\fie_{x}$ and $\Gamma^\psi_{y}$ have isomorphic dimension ranges''.
\end{rmk}

\begin{notation}
For simplicity, in the following proof we write $\Lambda_x$ instead of $\Gamma^\psi_x$, and $\Gamma_x$ instead of $\Gamma^\fie_x$ (for $x\in X$).
\end{notation}

\begin{proof} 
First of all, note that conditions \ref{condition 2} and \ref{condition 3} are equivalent by Lemma 
\ref{lemma GPS SOE-changing base point does not change dimension range} and Remark \ref{rmk: abstract isomorphism is equivalent to isomorphism between dimension ranges}.

\medbreak

Let $x$ be in $X$ and suppose that $\Gamma_x$ and $\Lambda_x$ have isomorphic dimension ranges. By continuity of $\fie$, the top point of a sequence of partitions $(\Xi_N)_N\in\N$
associated to $\Gamma_x$ is $\inv{\fie}(x)=y_1$. We also denote $y_2=\inv{\psi}(x)$. By Krieger's theorem (Theorem \ref{Krthm}) and Remark \ref{rmk : more than original Krieger is true}, there exists a homeomorphism $g$ such that $\Lambda_x=g\Gamma_x \inv{g}$, $g(x)=x$ and $g(y_1)=y_2$. This control on these two pairs of points ensures that $g$ is indeed an orbit equivalence.
Let $z\neq y_1$. Then for a large enough $N$, $z$ does not belong to the top of $\Xi_N$, hence there exists $h\in \Gamma_x$ 
and a neighbourhood $W$ of $z$ such that $\restr{h}{W}=\restr{\fie}{W} $. On the other hand $gh\inv{g}\in\Lambda_x\subset \llbracket\psi\rrbracket$,
so there exists a neighbourhood $V$ of $g(z)$ and an integer $n_0$ such that $\restr{gh\inv{g}}{V}=\restr{\psi^{n_0}}{V}$.
Finally we get:
\[\forall z'\in \inv{g}(V)\cap W \ g(\fie(z'))=g(h(z'))=gh\inv{g}(g(z'))=\psi^{n_0}(g(z')),\]
and we deduce that the cocycle $n$ is continuous at the point $z$, hence at every point except $y_1$. The situation being symmetric,
we also get that the other cocycle is continuous everywhere, except maybe in one point, and by definition $g$ is then a strong orbit equivalence between $\fie$ and $\psi$.

\medbreak

Conversely, suppose that $g$ realizes a strong orbit equivalence between $\fie$ and $\psi$. By replacing $\psi$ by $\inv{g}\psi g $, we can assume that
$g=id$. Then 
$$\forall x \in X \ \fie (x)=\psi^{n(x)}(x) \text{ and } \psi (x)=\fie^{m(x)}(x)$$  with $n$ and $m$ being continuous except maybe in $y$ and
$y'$ respectively. We set $x=\fie (y)$ and $x'=\psi (y')$, and construct $\Gamma_{x}$ associated to a sequence of K-R partitions $(\Xi_N)_{N\in\N}$
with bases $(B^N)_{N\in\N}$ and $\Lambda_{x'}$ associated to a sequence of K-R partitions $(\Xi'_N)_{N\in\N}$ with bases $(D^N)_{N\in\N}$. We want to show that 
$\accfonction{\overline{id}}{\bigslant{CO(X)}{\Gamma_{x}}}{\bigslant{CO(X)}{\Lambda_{x'}}}{Orb_{\Gamma_{x}}(A)}{Orb_{\Lambda_{x'}}(A)}$ 
is a bijection. 

By symmetry, if we prove that it is well defined, the assertion follows. We use the following lemma, that we will prove right after the end of the current proof:

\begin{lema}\label{lemma GPS SOE-be in Gamma orbit can be seen piecewise}
Let f be a homeomorphism and $A=\bigsqcup_{i=1}^r A_i$ a clopen partition of $A$. Assume that for all $i$, $f(A_i)\in Orb_{\Lambda_{x'}}(A_i)$. Then $f(A)\in Orb_{\Lambda_{x'}}(A)$.
\end{lema}

Let $A$ be a clopen set, and $\gamma\in\Gamma_{N,{x}}$, where $N$ is sufficiently large that $A$ is in $\langle\Xi_N\rangle$. 
We have to show that $\gamma(A)\in Orb_{\Lambda_{x'}}(A)$. 
Thanks to Lemma \ref{lemma GPS SOE-be in Gamma orbit can be seen piecewise}, it is sufficient to show that two atoms of the same tower of $\Xi_N$ are in the same $\Lambda_{x'}$-orbit, or, told differently, that $\fie(A)\in Orb_{\Lambda_{x'}(A)}$ for every atom $A$ of $\Xi_N$ which is not on the top of the partition.
For such an atom $A$, one has $y\notin A$, and so the cocycle $n$ is continuous on $A$ which is compact, hence $A$ can be partitioned into pieces on which $n$ is constant. Thus using Lemma \ref{lemma GPS SOE-be in Gamma orbit can be seen piecewise} once more, we can assume that 
$n$ is constant on $A$, equal to a certain $n_0$. 
We set $M=\abs {n_0}$. Now, for all $x$ in $X$ there exists a clopen neighbourhood $V_x$ of $x$ such that $\psi^{-M}(V_x),\ldots,\psi^{M+1}(V_x)$ are pairwise disjoint (see Remark \ref{rmk: finite images of a small neighbourhood are disjoint}).
Then by compactness of $A$ we get a partition $A=\bigsqcup_{i=1}^{r} A_i$, with each $A_i$ such that $\psi^{-M}(A_i),\ldots,\psi^{M+1}(A_i)$ are pairwise disjoint. Let $i\in\llbracket1,r\rrbracket$ be fixed, and choose $x_i\notin \bigcup_{j=-M}^M\psi^j(A_i)$ (for example $x_i\in \psi^{M+1}(A_i)$), so that none of the $\psi^j(x_i)$ is in $A_i$ for 
$j\in\llbracket-M,M\rrbracket$. That means that no atom contained in $A_i$ belongs to $\psi^j(D^i_k)$ 
(where the $D^i_k$'s are the bases of a sequence of K-R partitions $(\Xi'^i_k)_{k\in\N}$ associated to $\Lambda_{x_i}$) for $k$ large enough and $j\in\llbracket-M,M\rrbracket$, 
i.e every atom contained in $A_i$ is at distance at least $M$ from the top and the bottom of its tower, which means that 
$\fie(A_i)=\psi^{n_0}(A_i)\in Orb_{\Lambda_{x_i}}(A_i)=Orb_{\Lambda_{x'}}(A_i)$ (this last equality being due to Lemma 
\ref{lemma GPS SOE-changing base point does not change dimension range}).
Using Lemma \ref{lemma GPS SOE-be in Gamma orbit can be seen piecewise} one last time, we conclude that $\fie(A)\in Orb_{\Lambda_{x'}}(A)$, which concludes the proof.
\end{proof}

\begin{proof}[Proof of Lemma \ref{lemma GPS SOE-be in Gamma orbit can be seen piecewise}]
Let $N$ be large enough that for all $i$, $A_i$ and $f(A_i)$ are in $\langle\Xi'_N\rangle$ and there exists 
$h_i\in \Lambda_{N,{x'}}$ such that $f(A_i)=h_i(A_i)$. For all $k$, and for all $j\in \llbracket0,H'^N_k-1\rrbracket$, we define 
$h\in\Lambda_{N,{x'}}$ on an atom $D^N_{k,j}$ of $\Xi'_N$ as follows (see figure \ref{reduc_orbites}):
\begin{itemize}
 \item if $D^N_{k,j}\nsubseteq A\cup f(A)$, we set $h=id$ on $D^N_{k,j}$ ;
 \item if $D^N_{k,j}\subset A_i$ for some $i$, we set $h=h_i$ on $D^N_{k,j}$ ;
 \item if $D^N_{k,j}\subset f(A)\setminus A$, we find the least positive integer $m$ and an integer $k_m$ such that
$$f^{-m}(D^N_{k,j})\subset A\setminus f(A) \text{ and }  \psi^{k_m}(D^N_{k,j})=f^{-m}(D^N_{k,j})$$ and we set 
$h=\psi^{k_m}$ on $D^N_{k,j}$.
\end{itemize}

\begin{figure}[h]
    \centering
    \includegraphics[height=5cm]{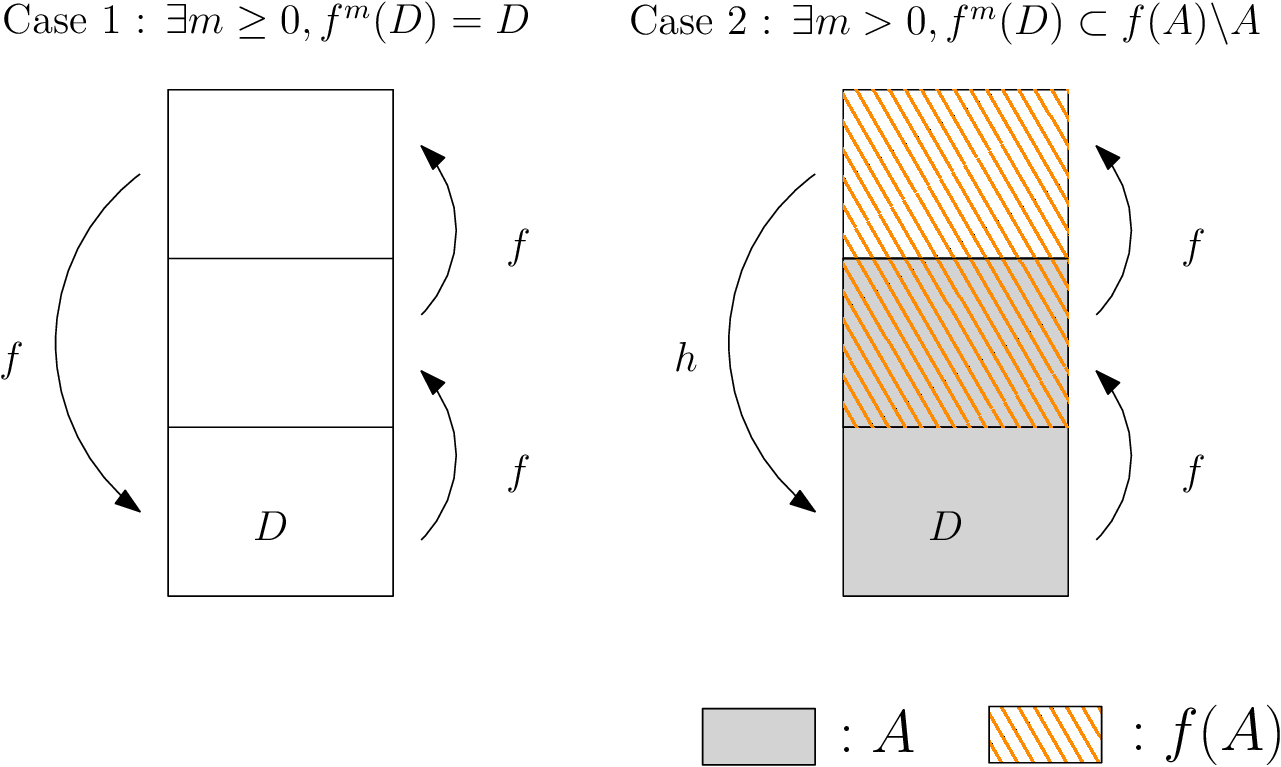}
    \caption{Construction of $h$, example with $m=2$}
    \label{reduc_orbites}
\end{figure}

\noindent It is clear that $h\in\Lambda_{N,{x'}}$ and $h(A)=\bigsqcup_{i=1}^r h_i(A_i)=\bigsqcup_{i=1}^r f(A_i)=f(A)$
\end{proof}

\section{Borel complexity of isomorphism of countable, locally finite simple groups}\label{section: Borel complexity}

In this section we study algebraically in some detail the countable, locally finite groups $\Gamma$ that we have used above. Since for $x,y\in X$, $\Gamma_x$ and $\Gamma_y$
are (spatially) isomorphic by a combination of Lemma \ref{lemma GPS SOE-changing base point does not change dimension range} and Theorem \ref{Krthm},
we denote them $\Gamma$ and do not specify which semi-orbit it preserves. Observing that those groups are not always simple, 
we study their commutator subgoups and show that they are simple; further, we point out that if the commutator subgroups of two of these groups 
are isomorphic, then the groups themselves are isomorphic.
It enables us to show that the strong orbit equivalence relation is, in a way we are going to recall precisely (Borel reducibility theory, 
cf section \ref{subsection: Borel reducibility}), ``simpler'' than the isomorphism relation between countable, locally finite simple groups. 
\medbreak
\subsection{Study of $D(\Gamma)$}\label{subsection: study of D(Gamma)}
We have claimed that $\Gamma^\fie$ is not always simple, so let us give an explicit example. Let $\sigma_3$ denote the 3-odometer on the
Cantor space  $X=\{0,1,2\}^\N$, and $N_\alpha$ be the clopen subset consisting of all sequences starting by the finite sequence $\alpha$. Then we have a sequence
of K-R partitions, each consisting of a single tower of basis $N_{0^n}$ for the $n$-th partition. The cutting and stacking process
is then very simple: the $n+1$-th tower is a stacking of three copies of the $n$-th tower. The element $\gamma$ of $\Gamma^{\sigma_3}$ 
defined by $$\restr{\gamma}{{N_0}}=\restr{{\sigma_3}}{{N_0}} \text{ and } \restr{\gamma}{{N_1}}=\restr{\inv{{\sigma_3}}}{{N_1}}$$
is clearly not in the commutator subgroup, since on $\Xi_n$ it decomposes as a product of $3^{n-1}$ transpositions, which belongs to
$\mathfrak{S}_{3^{n}}\backslash\mathcal{A}_{3^{n}}$.

We however have the following:
\begin{prop}
For every minimal homeomorphism $\fie$, $D(\Gamma^\fie)$ is simple.
\end{prop} 

To see that, it suffices to read carefully section 3 in {\cite{BM}}, where a similar result is established for commutator subgroups of topological full groups; 
and to observe that the argument given there works exactly the same way, replacing 
$\llbracket\fie\rrbracket$ by $\Gamma^\fie$. 

\begin{figure}[ht]
    \centering
    \includegraphics[height=6cm]{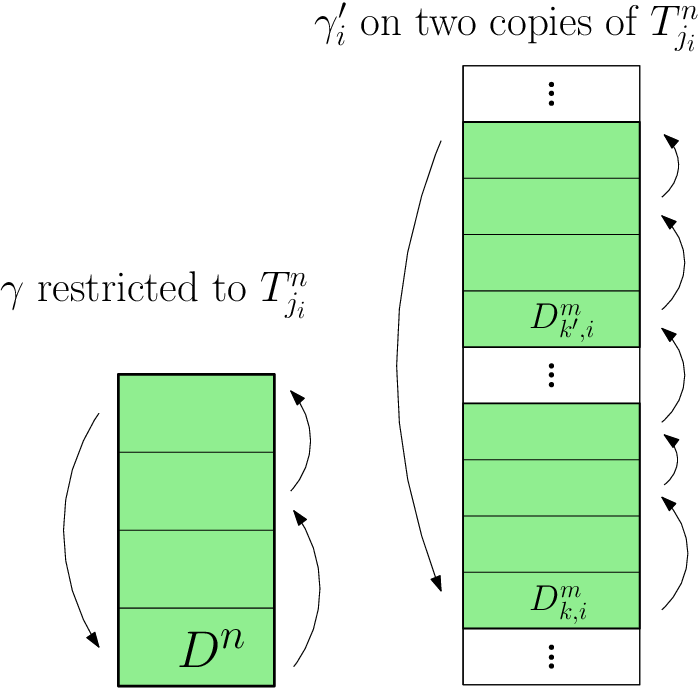}
    \caption{Construction of $\gamma'_i$ in Remark \ref{D(gamma) dense}}
    \label{gamma'i}
\end{figure}

\begin{rmk}\label{D(gamma) dense}
The only remark we maybe need to make, in order to follow the argument of \cite{BM}, is that $D(\Gamma)$ is dense in $\Gamma$. Indeed, let $\gamma\in\Gamma$, and $A_1,\ldots,A_r$ be disjoint clopen sets. There are some integers $n$ and $m$ such that $A_i$ belongs to $\langle\Xi_n\rangle$ for all $i$ and towers of $\Xi_m$ are high enough that they must contain at least two copies of some tower of $\Xi_n$. Let us assume that the tower $T^m_i$ contains two copies of 
$T^n_{j_i}$.
Let $D^m_{i,k}$ and $D^m_{i,k'}$ be distinct atoms of $T^m_i$ such that $D^m_{i,k}\cup D^m_{i,k'}\subset D^n\subset T^n_{j_i}$ (where $D^n$ is an 
atom of $T^n_{j_i}$, no matter which one).
Let $N=\min\{l>0\colon \gamma^l(D^n)=D^n\}$.
We define $\gamma'_i\in \Gamma_m$ by $\gamma'_i(D)=
            \begin{cases}
               D^m_{i,k'}       &\text{if } D=\gamma^{N-1}(D^m_{i,k}) \\
               D^m_{i,k}       &\text{if } D=\gamma^{N-1}(D^m_{i,k'})\\
               \gamma(D)        &\text{elsewhere}
            \end{cases}$.
            
\noindent Figure \ref{gamma'i} shows how it works with $N=3$.
The element $\gamma'_i$ is constructed in such a way that, viewing $\restr{\gamma}{{T^m_i}}$ and $\restr{{\gamma'_i}}{{T^m_i}}$ as permutations 
of atoms of the tower $T^m_i$, exactly one of them belongs to the alternating group. Denote this restriction by $\overset{\sim}{\gamma_i}$. 
Then $\overset{\sim}{\gamma}=\bigsqcup_i \overset{\sim}{\gamma_i}\in D(\Gamma)$ and acts the same way as $\gamma$ on $\langle\Xi_n\rangle$, in particular on every $A_i$.
\end{rmk}
\smallbreak
We now need the following:

\begin{prop}\label{prop: every isom is spatial}
Let $\fie$ and $\psi$ be two minimal homeomorphisms. Then any group isomorphism $\fct{\alpha}{D(\Gamma^\fie)}{D(\Gamma^\psi)}$ is  spatial.
\end{prop}

The proof is very much inspired by \cite{Med}, though less technical. It uses as its principal ingredient the reconstruction Theorem 384D of 
\cite{Fr}.
In order to use it we need the following definitions:
\begin{deff}
\begin{itemize}
    \item An open set $A$ is called \textit{regular} if $A=int(\overline{A})$. The set of all regular open sets forms a Boolean algebra we denote
    by $RO(X)$. Obviously clopen sets are regular, and so $CO(X)\subset RO(X)$
    \item A group $G\leq\text{Homeo}(X)$ is said to \textit{have many involutions} if for any regular open set $A$, there exists an involution 
    $g\in G$ whose support is included in $A$.
\end{itemize} 
\end{deff}

\begin{rmk}\label{manyinvol}
The "support of $g$" is here defined by $int(\{x\in X\colon g(x)\neq x\})$, but since in our case $g$ always belongs to the topological 
full group it makes no difference with the usual definition.\newline
Note that for every minimal homeomorphism $\fie$, $D(\Gamma^\fie)$ has many involutions. Indeed, if $A$ is a regular open set, there exists a nonempty
clopen set $C\subset A$. Think of $\Gamma^\fie$ as being constructed from a sequence $\Xi$ of K-R partitions whose base point is $x_0\in X$. 
By minimality of $\fie$, for $n$ large enough there are at least 4 atoms in the tower $T^n_{i_0}$ of $\Xi_n$ containing $x_0$ which are 
contained in $C$, say $D^n_{i_0,k_0}, D^n_{i_0,k_1}, D^n_{i_0,k_2}, D^n_{i_0,k_3}$, with $k_j<k_{j+1}$. Then we can easily define an 
involution $g\in D(\Gamma_n)$ whose support is contained in $C$, and thus in $A$ (for example the double transposition 
$(D^n_{i_0,k_0} \ D^n_{i_0,k_1})(D^n_{i_0,k_2} \ D^n_{i_0,k_3})$).
\end{rmk}

\begin{proof}
Let $\fct{\alpha}{D(\Gamma^\fie)}{D(\Gamma^\psi)}$ be an isomorphism.
Thanks to remark \ref{manyinvol}, we now know that $D(\Gamma^\fie)$ and $D(\Gamma^\psi)$ both have many involutions.
Theorem 384D applies and gives us an automorphism of Boolean algebras 
$$\fct{\Lambda}{RO(X)}{RO(X)}$$ such that $\alpha(g)(V)=\Lambda g\inv{\Lambda}(V)$ 
for all $g\in D(\Gamma^\fie)$ and $V\in RO(X)$. If we can show that $\Lambda(CO(X))=CO(X)$, then $\Lambda$ is induced by a homeomorphism of $X$ 
and the proof is over. We proceed to explain why this is true.

We note that $CO(X)$ is generated by the supports of the involutions in $D(\Gamma)$, where $\Gamma$ stands for either 
$\Gamma^\fie$ or $\Gamma^\psi$ (or more generally for any $\Gamma$ associated to a minimal homeomorphism).
Indeed, take a clopen set $C$, and look at $\Gamma$ as constructed out of a sequence of K-R partitions $\Xi$. For $n$ large enough, $C$ belongs to $\langle\Xi_n\rangle$ and every tower of $\Xi_n$ has an height greater than 7. Now, given an atom $A$ of $\Xi_n$, it is easy to construct 
two permutations $\gamma_1,\gamma_2\in D(\Gamma)$ such that $\supp(\gamma_1)\cap \supp(\gamma_2)=A$, indeed it suffices to take two double transpositions
defined on $A$ and three other atoms each, pairwise disjoint, which is possible because of the height of the tower being greater than 7.
The proof is now over by applying the following lemma, that follows from the proof of the Theorem 384D of \cite{Fr}:
\begin{lema}
If $g\in D(\Gamma^\fie)$ is an involution, then $\supp(\alpha(g))=\Lambda(\supp(g))$.
\end{lema}
Indeed, this shows that $\Lambda(CO(X))\subset CO(X)$, and since the situation is symmetric we obtain as desired the equality $\Lambda(CO(X))=CO(X)$.
\end{proof}

Let us sum up what we know: given two minimal homeomorphisms $\fie$ and $\psi$, we have
\[
\begin{aligned}
     &D(\Gamma^\fie), \ D(\Gamma^\psi) \text{ are isomorphic as abstract groups }\\
     &\iff D(\Gamma^\fie), \  D(\Gamma^\psi) \text{ are spatially isomorphic }\\
     &\iff D(\Gamma^\fie), \  D(\Gamma^\psi) \text{ have isomorphic dimension ranges, (Theorem \ref{Krthm})}\\
     &\iff \Gamma^\fie, \  \Gamma^\psi \text{ have isomorphic dimension ranges, (Remark \ref{D(gamma) dense})}\\
     &\iff \fie, \  \psi \text{ are strong orbit equivalent, (Theorem \ref{gps1})}
\end{aligned}
\]

\subsection{Isomorphism on the space of countable, locally finite simple groups is a universal relation}\label{subsection: Borel reducibility}

Let us start by recalling the definition of Borel reducibility from the introduction :
\begin{deff}
Let $E,F$ be equivalence relations on standard Borel spaces $X,Y$ respectively. $E$ is said to be \emph{Borel reducible to $F$}, and we write 
$E\leq F$ if there exists a Borel map $\fct{f}{X}{Y}$ such that $$\forall x,x'\in X \ xEx'\iff f(x)Ff(x').$$ We call such a map $f$ a 
\emph{Borel reduction from $E$ to $F$}. If $E\leq F$ and $F\leq E$, we say that \emph{$E$ and $F$ are Borel bireducible}.
\end{deff}

We also recall that there exists a universal equivalence $E_\infty$ arising from an action of $S_\infty$ on a standard Borel space, that we denote $E_\infty$. A theorem of Melleray gives us a particular realization of $E_\infty$ :

\begin{thm}[{\cite{Mel20}}]
The equivalence relation of strong orbit equivalence of minimal homeomorphisms on the Cantor space is Borel bireducible to $E_\infty$.
\end{thm}

We have obtained in the previous subsection (see the series of equivalences at the end of subsection \ref{subsection: study of D(Gamma)}) that strong orbit equivalence between two minimal homeomorphisms $\fie$ and $\psi$ is characterized by the isomorphism of $D(\Gamma^\fie)$ and $D(\Gamma^\psi)$, which are countable, simple and locally finite.
We need to exhibit a Borel way to associate to a minimal homeomorphism $\fie$ its countable, locally finite simple group $D(\Gamma_{x_0}^\fie)$.
There are plenty of ways to do it and the one we have chosen requires a lot of steps; but it is quite easy to check that each of those steps is Borel.
Let $\homeomin(X)$ be the Borel subset of $\homeo(X)$ that consists of minimal homeomorphisms of $X$.

First of all we define 
$$\fct{EnumTfg}{\homeomin(X)}{\homeo(X)^\omega}$$
such that $EnumTfg(\fie)$ is an enumeration of $\llbracket\fie\rrbracket$.

Let $(n^i_1,\ldots,n^i_{k_i},A^i_1,\ldots A^i_{k_i})_{i\in\N}$ be an enumeration of $\bigsqcup_{k\in\N} \N^k\times CO(X)^k$. We define $EnumTfg$ as follows:
for all $i\in\N$, if $\bigsqcup_{j=1}^{k_i} A_j^i=X$ and $\bigsqcup_{j=1}^{k_i} \restr{\fie^{n^i_j}}{A^i_j}\in\llbracket\fie\rrbracket$, then 
$EnumTfg(\fie)(i)=\bigsqcup_{j=1}^{k_i} \restr{\fie^{n^i_j}}{A^i_j}$, else, put it equal to $id$. This defines an enumeration of $\llbracket\fie\rrbracket$
with repetitions. Note that one can always make a sequence injective in a Borel way by removing duplicates.

Now that we have encoded $\llbracket\fie\rrbracket$, we define 
$$\fct{Enum\Gamma_{x_0}}{\homeomin(X)}{\homeo(X)^\omega}$$ such that 
$Enum\Gamma_{x_0}(\fie)$ enumerates $\Gamma^\fie_{x_0}$. 
To do this, we introduce the function 
$$\fct{IsIn\Gamma}{\homeo(X)\times\homeomin(X)}{\homeo(X)}$$ that associates $\alpha$ to $(\alpha, \fie)$ if it belongs
to $\Gamma^\fie_{x_0}$, and $id$ otherwise. The condition ``$\alpha$ belongs to $\Gamma^\fie_{x_0}$'' can be written 
$$\begin{array}[t]{lcrl}
    \alpha\in\llbracket\fie\rrbracket 
    & \text{ and } \forall n\in\N \exists m\in\N \alpha(\fie^n(x_0))=\fie^m(x_0) \\
    &  \text{ and } \forall n\in\N \exists m\in\N \alpha(\fie^m(x_0))=\fie^n(x_0)
\end{array} $$
whence $IsIn\Gamma$ is a Borel function.
Then it is clear that $Enum\Gamma_{x_0}$ defined by $$Enum\Gamma_{x_0}(\fie)_i=IsIn\Gamma((EnumTfg(\fie)_i,\fie))$$ 
is a Borel function.

It remains to encode $D(\Gamma)$. Since one can define Borel functions $Commu$ and $GenBy$ which enumerate respectively all commutators of a given sequence of homeomorphisms
and the group generated by it, we can define 
$$EnumD\Gamma=GenBy\circ Commu\circ Enum\Gamma_{x_0}$$
It is a Borel reduction of the strong orbit equivalence relation to the isomorphism relation on countable, locally finite simple groups, and we have proved the following :

\begin{thm}
The relation of isomorphism of countable, locally finite, simple groups is a universal relation arising from a Borel action of $S_\infty$.
\end{thm}

So in particular, in the case of Borel reducibility it is as complicated to classify simple, locally finite groups as it is to classify countable groups.

\subsection*{Acknowledgements}
The author would like to thank an anonymous
referee for their work and helpful comments, and also Julien Melleray for
suggesting this approach, in particular the link with Borel complexity.

\end{document}